\documentclass[11pt]{amsart}
\usepackage[english]{babel}

\usepackage{amsmath,amsthm,amsfonts,amssymb, mathrsfs, mathbbol}
\usepackage[foot]{amsaddr}
\usepackage{mathtools}
\usepackage{color}

\usepackage{graphics}
\usepackage{booktabs}

\def\text{\mbox}

\def\div{\mathrm{div}\,}

\def\dps{\displaystyle}

\def\R{{\mathbb R}}

\def\N{{\mathbb N}}

\def\P{{\mathbb P}}

\def\D{{\mathscr D}}

\def\CC{{\mathscr C}}

\def\pt{\partial}
\def\omm{\Omega}

\def\<{\langle}
\def\>{\rangle}
\def\inclus{{\hookrightarrow}}

\def\grad{\nabla}

\def\bve{\;|\;}

\def\fleche{\rightarrow}

\newenvironment{remark}{ {\sc Remark -- } }   {\\}  
\newtheorem{theorem}{Theorem}[section]
\newtheorem{proposition}[theorem]{Proposition}
\newtheorem{lemma}[theorem]{Lemma}

\newcommand{\bfgreek}[1]{\bm{\@nameuse {up#1}}}

\def\hv{{\hat{v}}}

\def\Rt{\R^3}

\newcommand{\vecc}[1]{ {\boldsymbol #1}}

\def\v{\vecc{v}}

\def\zero{\vecc{0}}
\def\a{\vecc{a}}
\def\e{\vecc{e}}

\def\h{\vecc{h}}
\def\x{\vecc{x}}
\def\y{\vecc{y}}

\def\xs{\vecc{x}_\star}
\def\omm{\ommega}
\def\oms{\ommega_\star}
\def\Wh{W_h}

\def\curl{{\boldsymbol{\rm curl}\,}}

\def\En{{\mathscr E}}

\newenvironment{ourproof}[1]{ \begin{quote}{\sc Proof  #1} -- }   {  $\blacksquare$  \end{quote}   }

\def\Tauh{{{\mathcal T}_h}}
\def\Taush{{\mathcal T}_h^{\star}}
\def\Taussh{{\mathcal T}_h^{\star\star}}

\def\omz{\Omega_0}
\def\bomz{\overline{\Omega}_0}
\def\oms{\Omega_\star}
\def\ominft{\Omega_\infty}
\def\bominft{\overline{\Omega}_\infty}

\def\Ks{K_\star}

\def\cst{c^\star}
\newcommand{\nrm}[1]{|{#1}|}

\def\eH{\vecc{H}_e}
\def\vH{\vecc{h}}
\def\Mg{\vecc{M}}
\def\Mz{\vecc{M}_0}
\def\n{\vecc{n}}

\def\EE{{\mathscr E}}
\def\EEST{{\mathscr E}_{sf}}
\def\EESTH{{\mathscr E}_{sf}}
\def\indom{\chi_\omm}
\def\bomm{\overline{\omm}}
\def\St{{\mathbb S}^2}
\def\mz{\Mg_0}

\def\omm{\Omega}
\def\bom{\overline{\omm}}

\def\RZ{R_0}
\def\RAY{{r_0}}
\renewcommand{\leq}{\leqslant}
\renewcommand{\geq}{\geqslant}

\author[T. Z. Boulmezaoud]{Tahar Z. BOULMEZAOUD$^{1, 2, 3}$}
\address{\rm  $^1$ Universit\'e Paris-Saclay, UVSQ, LMV, Versailles, France.}
\address{\rm  $^2$  Department of Mathematics and Statistics, University of Victoria, Victoria, British Columbia, Canada.}
\email{tahar.boulmezaoud@uvsq.fr}
\author[K. Kaliche]{ Keltoum KALICHE$^{1, 3}$}
\address{\rm  $^3$ University Kasdi Merbeh, Ouargla, Algeria. }
\email{keltoumkaliche@yahoo.fr}

\DeclareRobustCommand{\subtitle}[1]{\\#1}
\title[Stray field  computation by inverted finite elements]{\Large Stray field computation by inverted finite elements: a new method in micromagnetic simulations}



\begin{document}

\begin{abstract}
In this paper, we propose a new method for computing the stray-field and 
the corresponding energy for a given magnetization configuration. Our approach is based on the use of inverted finite elements and does not need any truncation. After analyzing the problem in an appropriate functional framework, we describe the method and we prove its convergence. We then   display some computational results which  demonstrate its efficiency  and confirm its full potential.
 \end{abstract}

\begingroup
\def\uppercasenonmath#1{} % this disables uppercasing title
\let\MakeUppercase\relax % this disables uppercasing authors
\maketitle
\endgroup

\renewcommand{\subtitle}[1]{}
\section{Introduction}
In micromagnetics, the structure of a magnetic body $\omm\subset \R^3$ if often described by the magnetization $\Mg$, which is a vector field defined over $\omm$ and minimizing the Landau-Lifschitz functional. In terms of dimensionless variables, the latter energy can be written into the form
$$
\EE(\Mg) = \alpha \int_{\omm} |\grad \Mg|^2 + \int_{\omm} \varphi(\Mg) dx + \frac{1}{2} \int_{\R^3} |\grad u|^2 dx - \int_{\R^3} \eH.\Mg dx, 
$$
where $\varphi \geq 0$ is a function describing the orientation of the magnetization, $\alpha > 0$ is a real parameter, 
$\eH$ the external magnetic field and $u$ is magnetostatic  potential. The latter quantity is related to the stray-field (or the magnetic induction)  $\vH$ by 
\begin{equation}\label{potent_vh}
\vH = - \grad u. 
\end{equation}
The existence of the scalar potential $u$ comes from Maxwell's equation
\begin{equation}\label{hrot_equa}
\curl \vH = \zero \mbox{ in } \R^3.
\end{equation}
Moreover, the stray-field $\vH$ and the magnetization $\Mg$ are related by the equation
\begin{equation}\label{relationMgHm}
\div(\vH + \Mg \indom) = 0 \mbox{ in } \R^3,
\end{equation}
where $\indom$ stands for characteristic (or indicator) function of $\omm$.  Rewritting \eqref{relationMgHm} in terms of $u$ and $\Mg$ gives the well known equation
\begin{equation}\label{main_equa}
\Delta u = \div (\Mg \indom)   \mbox{ in } \R^3. 
\end{equation}
This equation can also be written into the form
\begin{equation}\label{decompoLaplSyst}
\left\{
\begin{array}{rcll}
\Delta u &= & \div \Mg &\mbox{ in } \omm,  \\
\Delta u &= & 0 & \mbox{ in }\R^3 \backslash \overline{\omm},  \\
\dps{[u]} &=& 0 &  \mbox{ on } \pt \omm, \\
\dps{ \left[ \frac{\pt u}{\pt n} \right] }&=& - \Mg.\n & \mbox{ on } \pt \omm, 
\end{array}
\right.
\end{equation}
where $\n$ is the exterior normal on $\pt \omm$. 
In addition, the magnetization $\Mg$ is subject to the Heisenberg-Weiss condition
\begin{equation}\label{heisenweiss}
| \Mg |   \mbox{ is constant in } \omm.
\end{equation}
It is well known that calculating the stray-field $\vH$ and the corresponding energy
\begin{equation}
\EEST(u) = \frac{1}{2} \int_{\R^3} |\vH|^2 dx, 
\end{equation}
from the magnetization $\Mg$ is one of the most  important steps in studying micromagnetic configurations of a body $\omm$. We may observe that a  consequence of identity \eqref{main_equa} is that the stray-field energy also writes 
\begin{equation}
\EEST(u) = \frac{1}{2} \int_{\R^3} |\grad u|^2 dx =  - \frac{1}{2}  \int_{\omm} \Mg.\h dx. 
\end{equation}
 (see also the weak formulation of \eqref{main_equa} hereafter). \\
 In the existing litterrature,
one can find mainly two categories  of methods. In the first category the calculation of $u$ and  $\vH$ is often based on solving  the elliptic partial differential equation
\eqref{main_equa}. In that case, the computational domain is often  truncated and approximation is  done in a sufficiently large bounded region (see, e. g., \cite{Berkov1}, \cite{LuskinMA}, \cite{KoehlerFred1, KoehlerFred2},   \cite{carstensen2001} and  \cite{Prohlbook}, \cite{hanbao_poisson1999} and \cite{hanbao_siam2000}). In the second category of methods, the  approach consists to evaluate $u$ using the  integral  formula (see,  e. g., \cite{ExlAuz})
\begin{equation}\label{green_formula}
u(\x) = \frac{1}{4\pi} \int_{\omm} \frac{(\x-\y).\Mg(\y)}{|\y-\x|^3} d\y. 
\end{equation}
Among methods using formula \eqref{green_formula}, one can mention methods based on the Fast Fourier transfor and fast Multipole methods (see, e. g.,  \cite{BlueShen}, \cite{longOng}), $H$-matrix techniques (\cite{Popovi}) or direct integration methods (see, e. g., \cite{ExlAuz}, \cite{labbe}).  \\
In this paper, the focus is on computing the stray-field $\vH$ from the magnetization $\Mg$  by a novel approach based on the use of inverted finite element method (IFEM). IFEM was first introduced by Boulmezaoud in \cite{boulmezaoudm2an}  for solving elliptic problem in unbounded domains without any truncation. In the context of equation \eqref{main_equa}  considered here, the domain of computation $\R^3$ is considered in its entirety.  The deployment of IFEM  is based on a weak formulation of \eqref{main_equa}  in an appropriate weighted space. \\
$\;$\\
The paper is organized as follows. In section \ref{etude_pb}, we employ some weighted function spaces to study equation  \eqref{main_equa}, completed with asymptotic conditions when $|\x| \fleche +\infty$.  In particular, we give some details about the behavior at large distances  and about the   smoothness of the solution and of its derivatives.  Section \ref{ifem_presentation} is devoted to an outline of IFEM.   After giving the general lines of the method, we prove its convergence in the context of 
equation \eqref{main_equa}. In the last  section, we give some numerical results obtained with a 3D code. 

\section{Preliminaries. Well posedness of the problem}\label{etude_pb}
In the sequel, $\omm$ denotes an open and  connected  subset of $\R^3$ having a lipschitzian boundary, not necessarily bounded (althought in physical applications $\omm$ is often a bounded domain). Let $\Mg$ be a vector field defined over $\omm$.  From a strictl mathematical point, unless otherwise indicated, we only assume that 
\begin{equation}\label{assump_h}
\int_{\omm} |\Mg|^2 dx < \infty.
\end{equation}
Assumption \eqref{assump_h} is obviously valid when $\omm$ is bounded  (or has a finite  volume) and when $|\Mg|$ is satisfying the Heisenberg-Weiss constraint \eqref{heisenweiss}. In that case
$$
\|\Mg\|^2_{L^2(\omm)^3} = \int_{\omm} |\Mg|^2 dx = |\Mg|^2  |\omm| < +\infty. 
$$
We now come back to  equation \eqref{main_equa}. Without going into the technicalities of Poisson equation, it can be seen that existence and uniqueness of solutions to \eqref{main_equa} 
depend on the required behavior at large distances, that is when $|\x| \fleche + \infty$. To illustrate this, one may observe that polynomial growth of solutions at large distances  should be excluded,  otherwise uniqueness may be lost  since  harmonic polynomials  can be added to any solution of  \eqref{main_equa} (see, e. g., \cite{giroire}, \cite{amrouche1}). Fortunately, in the current context, $u$ must fulfill the physical constraint
\begin{equation}\label{asympto_grad}
\ \int_{\R^3} |\grad u|^2 dx < \infty, 
\end{equation}
which means that $\vH = -\grad u$ has a finite energy.  In view of Hardy's inequality (see, e. g., \cite{amrouche1}), it is natural to require that
\begin{equation}\label{asympto_u}
\int_{\R^3} \frac{ |u|^2}{|\x |^2+1} dx < \infty.  
\end{equation}
For this reason, we need to introduce some  weighted function spaces. For all integers $\ell \geq 0$ and $m \geq 0$, 
  $W^m_\ell(\R^3)$ stands for the space of all the functions satisfying 
 $$
\forall |\lambda| \leq m, \; (1+|\x|^2)^{(\ell+|\lambda|-m)/2} D^\lambda v \in L^2(\Rt).
$$
This space is endowed with the norm 
\begin{equation}
\|v\|_{W^m_\ell(\Rt)} = \left( \sum_{|\lambda| \leq m} \int_{\Rt}  (|\x |^2+1)^{|\lambda|+\ell-m}   |D^\lambda v|^2dx\right)^{1/2}. 
\end{equation}
In other words,  elements of $W^1_0(\Rt)$  are those  functions satisfying  \eqref{asympto_grad} and \eqref{asympto_u}.  We may observe that
non vanishing polynomial functions do not belong to $W^1_0(\Rt)$. 
Before continuing with problem  \eqref{main_equa}, let us recall that for any function $v \in W^1_\alpha(\R^3)$,  $\alpha \in \R \backslash \{-1/2\}$,  one has (see, e. g., \cite{alliot_these})
\begin{equation}\label{prop_asympt_in_wei}
\lim_{|\x| \fleche + \infty} |\x|^{\alpha + 1/2} \|u(|\x|, .)\|_{L^2(\St)}   = 0, 
\end{equation}
where $\St$ is the unit sphere of $\R^3$ and 
\begin{equation}\label{defin_normSt}
\|u(|\x|, .)\|^2_{L^2(\St)}  = \int_{\St} |u(|\x|, \sigma )|^2 d\sigma. 
\end{equation}
Equation \eqref{main_equa}, completed with asymptotic conditions  \eqref{asympto_grad} and \eqref{asympto_u}, can be written into the variational form: find $u \in W^1_0(\R^3)$ such that
\begin{equation}\label{forma_weak}
\forall v \in W^1_0(\R^3), \; \int_{\R^3} \grad u. \grad v dx = \int_{\omm} \Mg.\grad v dx. 
\end{equation}
We have the following result
\begin{proposition}[Well posedness]\label{well_posed_prop}
Suppose that assumption \eqref{assump_h} holds true. Then, \eqref{forma_weak}, and consequently \eqref{main_equa},  has one and only one solution in $W^1_0(\Rt)$. Moreover, the following estimates hold
\begin{eqnarray}
\|(|\x|^2+1)^{-1/2} u\|_{L^2(\Rt)} &\leq & 2  \|\Mg\|_{L^2(\omm)},  \label{estima_u0}\\
 \; \|\grad u\|_{L^2(\Rt)^3}& \leq&   \|\Mg\|_{L^2(\omm)}. \label{estima_grad}
\end{eqnarray}
\end{proposition}
 We should note immediately that the decay of $u$ at large distances is faster than in estimates  \eqref{estima_u0} and \eqref{estima_grad}. Actually, $u \in L^2(\R^3)$ and $(|\x|^2+1)^{1/2} \grad u\in L^2(\R^3)^3$ as it will be stated in Proposition \ref{asympto_behave} hereafter. \\
\begin{ourproof}{of Proposition \ref{well_posed_prop}}
Let us first recall the classical Hardy inequality
\begin{equation}
\forall v \in \D(\Rt), \; \int_{\Rt}\frac{|u|^2}{|\x|^2} dx   \leq  4  \int_{\Rt}{|\grad u|^2} dx.
\end{equation}
Thus,
\begin{equation}\label{hardy_inequa2} 
\forall v \in \D(\Rt), \; \int_{\Rt}\frac{|u|^2}{|\x|^2+1} dx   \leq 4  \int_{\Rt}{|\grad u|^2} dx.
\end{equation}
By density of $ \D(\Rt)$ in $W^1_0(\Rt)$  (see \cite{hanouzet}), the last inequality remains valid 
for $v \in W^1_0(\Rt)$.  It follows that the bilinear form on the left hand side of
\eqref{forma_weak} is coercive. The linear form on the right hand side of \eqref{forma_weak} satisfies
$$
\left| \int_{\omm} \Mg.\grad v dx \right| \leq  \| \Mg \|_{L^2(\omm)^3}.\| \grad v \|_{L^2(\R^3)^3}.
$$
Existence and uniqueness follow from Lax-Milgram theorem. Moreover, taking $v=u$ in 
\eqref{forma_weak}  gives estimate \eqref{estima_grad}. Combining with inequality 
\eqref{hardy_inequa2}  gives \eqref{estima_u0}. 
\end{ourproof}
Since the right hand side of \eqref{forma_weak} is in a divergence form,  we also get the following result
\begin{proposition}[Asymptotic behavior]\label{asympto_behave}
Let $u \in W^1_0(\R^3)$ be solution of \eqref{main_equa}. Then, 
\begin{enumerate}
\item $u \in L^2(\R^3)$, 
\item $(1+|\x|^2)^{1/2} \grad u \in L^2(\R^3)^3$. 
\item  $\lim_{|\x| \fleche + \infty} |\x|^{3/2}  \|u(|\x|,  .)\|_{L^2(\St)} = 0$ where $\|u(|\x|, .)\|_{L^2(\St)} $ is defined by \eqref{defin_normSt}. 
\end{enumerate}
\end{proposition}
\begin{ourproof}{ of Proposition \eqref{asympto_behave}} 
Let us  prove that $u \in W^1_1(\R^3)$.  This is a direct consequence
of the following lemma  which is a particular case of 
a more general result proven in \cite{amrouche1} (Theorem 2.16):
\begin{lemma}\label{delta_op_lem}
Let $m \geq 0$ be an integer. Then, the Laplace operator  
$$ \Delta : W^{m+1}_{m+1}(\R^{3}) \mapsto W^{m-1}_{m+1}(\R^{3})\perp \R$$  is an isomorphism.
\end{lemma}
Here  $W^{-1}_{1}(\R^{3})$ stands for the dual space of $W^{1}_{-1}(\R^{3})$. 
 It may be noted at this point that 
constant functions belong to $W^{1}_{-1}(\R^{3})$.  By $W^{-1}_{1}(\R^{3})\perp \R$ we mean the space of functions
 $ f \in W^{-1}_{1}(\R^{3})$ satisfying 
 \begin{equation}\label{cond_ortho}
 \langle f, 1 \rangle_{W^{-1}_1(\R^3), W^{1}_{-1}(\R^3)} = 0. 
 \end{equation}
In the context of equation \eqref{main_equa}, the right hand side is $f = \div (\Mg\indom)$. Thus,
$f \in W^{-1}_1(\R^3)$ (since, obviously, $\Mg\indom \in W^0_{1}(\R^3)^3$) and \eqref{cond_ortho}
is automatically fulfilled.  We conclude that $u \in W^1_1(\R^3)$, thanks to Lemma \ref{delta_op_lem}. This ends the proof of the two first assertions.
The third assertion follows from property \eqref{prop_asympt_in_wei}. \\
\end{ourproof}
 \begin{remark}
 One can also prove that $u \in L^2(\R^3)$ by means of the Fourier transform. 
 \end{remark}

 \begin{proposition}[Regularity]\label{regul_propo}
 Assume that $\omm$ is a bounded open set of $\Rt$ with a $\CC^{1,1}$ boundary and  that 
  \begin{equation}
 \Mg \in L^2(\omm)^3, \; \div \Mg \in L^2(\omm) \mbox{ and } \Mg.\n \in H^{1/2}(\pt \omm).
 \end{equation}
Let $u \in W^1_0(\R^3)$ be solution of \eqref{main_equa}. Then, 
 \begin{enumerate}
 \item $u_{|\omm} \in H^2(\omm)$,
 \item $u_{|\R^3 \backslash \bomm} \in W^2_2( \Rt \backslash \bomm)$, that is
 $$
\forall 1 \leq i, j \leq 3, \; (|\x|^2+1) \frac{\pt^2 u}{\pt x_i \pt x_j} \in L^2(\Rt \backslash \bomm).
 $$
 \item[(c)] If $\Mg.\n = 0$ on $\pt \omm$, then $u \in W^2_2(\R^3)$.
 \end{enumerate}
 \end{proposition}
 \begin{ourproof}{ of Proposition \eqref{regul_propo}.}
Let $u_0 \in H^2(\omm)$ such that 
$$
u_0 = 0  \mbox{ and } \frac{\pt u_0}{\pt n} =  \Mg.\n \mbox{ on } \pt \omm.
$$
Set  
$$
U = \left\{
\begin{array}{ll}
 u - u_0 & \mbox{ in } \omm, \\
 u &\mbox{ in } \R^3 \backslash \bomm.
 \end{array}
\right.
$$ 
Since $[U] = 0$ on $\pt \omm$, we easily deduce that $ U \in W^1_0(\Rt)$. Moreover, we have
$$
[\frac{\pt U}{\pt n}] = 0 \mbox{ on } \pt \omm, 
$$
Thus, 
$$
\Delta U = (\div \Mg - \Delta u_0) \chi_{\omm} \in W^0_2(\Rt). 
$$
The right hand side of this equation satisfies
$$
\< \div \Mg - \Delta u_0, 1 \> = \int_{\omm} \div (\Mg-\grad u_0) dx=  \int_{\omm} (\Mg.\n-\frac{\pt u_0}{\pt n}) d\sigma = 0. 
$$ 
In view of Lemma \ref{delta_op_lem}, we deduce that $U \in W^2_2(\Rt)$. By restrincting to $\omm$ and to ${\Rt \backslash \bomm}$ we get  $u_{|\omm} = U_{|\omm} + u_0  \in H^2( \omm)$ and $u_{|\Rt \backslash \bomm} = U_{|\Rt \backslash \bomm} \in W^2_2(\Rt \backslash \bomm)$. Suppose now that 
  $\Mg.\n = 0$ on $\pt \omm$. Then, 
 $\div(\Mg \chi_\omm) \in L^2(\R^3)$.  Since $\div(\Mg \chi_\omm) = 0$ in $\R^2 \backslash \overline{\omm}$, we also deduce that
$\div(\Mg \chi_\omm)  \in W^0_2(\R^3)$. In view of Lemma \ref{delta_op_lem}, we deduce that $u \in W^2_2(\Rt)$.  This ends the proof of Proposition \ref{regul_propo}. 
 \end{ourproof}

Let us finish this section with an observation. In view of equations \eqref{decompoLaplSyst} the potential $u$ can be as
written as
\begin{proposition}\label{decompos_solution}
Suppose that $\Mg \in H(\div; \omm)$. Let $u \in W^1_0(\R^3)$  be the unique solution of \eqref{main_equa}. Then, 
\begin{equation}
u = u_0 + u_1, 
\end{equation}
 where $u_0  \in W^2_1(\R^3)$ is the unique solution of the Poisson equation 
 \begin{equation}\label{solu_uzsepare}
 \Delta u_0 = \widetilde{\div \Mg} \mbox{ in } \R^3,
  \end{equation}
  where $\widetilde{\div \Mg} $ designates the extension of ${\div \Mg} $  by zero outside $\omm$, while
 $u_1 \in W^1_0(\R^3)$ is the unique solution of the system
 \begin{equation}\label{formula_uone}
\int_{\R^3} \grad u_1 . \grad v dv = \<\Mg . \n, v\>_{H^{-1/2}(\pt \omm), H^{1/2}(\pt \omm)} \mbox{ for all } v \in W^1_0(\R^3). 
  \end{equation}
  Moreover,  $u_0\in W^2_2(\R^3)$ iff $ \<\Mg . \n, 1\> = 0$. In that case $u_1\in W^1_1(\R^3)$. 
 \end{proposition}
 \begin{proof}
 Since $\div \Mg \in L^2(\omm)$, $\widetilde{\div \Mg}  \in L^2(\R^3) \inclus W^0_{-1}(\R^3)$. 
 Following the same argument as  in the proof Proposition \ref{well_posed_prop}, we easily deduce
 existence and uniqueness of $u_0$ and $u_1$ solutions of \eqref{solu_uzsepare} and \eqref{formula_uone}, respectively.
 From Lemma \ref{delta_op_lem}, we deduce that $u_0 \in W^2_2(\R^3)$ iff   $\widetilde{\div \Mg} \in W^0_2(\R^3) \perp \R$, that is
 $$
 \<\Mg . \n, 1\>=\int_{\omm} \div \Mg dx= 0. 
 $$
Since $u \in W^1_1(\R^3)$ (see Proposition \ref{asympto_behave}),
 we also have $u_1 = u-u_0 \in W^1_1(\R^3)$. 
  \end{proof}
  \begin{remark}
 Under assumptions of Proposition \ref{decompos_solution}, we can write (see also, e. g., \cite{ExlAuz})
 $$
 u = -{\mathcal N} (\div \Mg) +  {\mathcal V} (\Mg . \n), 
 $$
 where ${\mathcal N}$ is the Newton potential defined by
 $$
 {\mathcal N} w (\x) = \frac{1}{4\pi} \int_{\omm} \frac{w(\y)}{|\x-\y|} d\y, 
 $$
 while ${\mathcal V}$ is the single layer potential  defined by
  $$
    {\mathcal V}   \phi (\x) = \int_{\pt \omm} \frac{\phi(\y)}{|\x-\y|} d\sigma(\y). 
  $$
When $\<\Mg . \n, 1\> \ne 0$,  ${\mathcal N} (\div \Mg) $ and ${\mathcal V} (\Mg . \n)$ decreases more slowly 
than $u$ when $|\x| \to +\infty$. In fact, in view of Proposition \ref{decompos_solution}, $u_0 = - {\mathcal N} (\div \Mg) \in W^2_1(\R^3) \inclus W^1_0(\R^3)$ and
$u_1 = {\mathcal V} (\Mg . \n) \in  W^1_0(\R^3)$ while $u_{\R^3 \backslash \bom}\in W^2_2(\R^3 \backslash \bom) \inclus  L^2(\R^3 \backslash \bom)$. Indeed,
when $\Mg$ is sufficiently smooth, we have
$$
|u_0 (\x)| = O(\frac{1}{|\x|}), \; |u_1(\x)|= O(\frac{1}{|\x|}), \; \mbox{ when } |\x| \to +\infty, 
$$
while, in view of formula \eqref{green_formula}, we have
$$
|u(\x)| = O(\frac{1}{|\x|^2}), \;  \mbox{ when } |\x| \to +\infty. 
$$
\end{remark}
$\;$\\
At this stage, mathematical aspects concerning equation  \eqref{main_equa} are prepared. It remains to show the way in which this problem is discretized by inverted finite elements method. This will be done in the next section.  
\section{Inverted finite elements method}\label{ifem_presentation}
Inverted finite elements method was developed by Boulmezaoud \cite{boulmezaoudm2an}. We will tailor it here
for solving problem \eqref{main_equa}. The starting point consists to partition the whole space $\R^3$ into two subdomains
\begin{equation}\label{decompoRt}
{\R}^{3} = {\bomz}\cup {\bominft}. 
\end{equation}
Here $\omz$ is bounded region while $\ominft$ is an unbouded one. We should note immediately  that the bounded  $\omz$ is not intented to  be large. In particular, we do not rule out the possibility that $\omz = \emptyset$ and $\ominft = \R^3$. However, the following constraint is imposed to  $\ominft$ (or, indirectly, to $\omz = \R^3\backslash \overline{\ominft}$): $\ominft$ is the non-overlapping union  of a finite number of infinite  tetrahedra, that is
\begin{equation}\label{decompo_inft}
{\bominft} =  T_1 \cup T_2 \cup \dots \cup T_M, 
\end{equation}
with $T_1$,..,$ T_M$, $M \geq 1$,  are $M$ infinite tetrahedra satisfying the  assumptions
\begin{itemize}
\item $T_1$,..,$ T_M$  have a common fictitious vertex.  Subsequently, we assume that this common fictitious vertex is the origin. 
\item the intersection of  two arbitrary   infinite  tetrahedra $T_i$ and $T_j$, $1 \leq i < j \leq M$, is either the empty set, a whole edge (a half-line) or a whole unbounded face.         
\end{itemize}
The concept of infinite  tetrahedron and,  more generally, of infinite simplices, was introduced in \cite{boulmezaoudm2an}.  For the sake of  clarity, we recall here this concept in 3D configurations. Given four non-coplanar points $\a_0$, $\a_1$, $\a_2$ and $\a_3$ of the euclidian affine space $\R^3$, define  the infinite tetrahedon $T$ whose vertices are  $\a_0$, $\a_1$, $\a_2$ and $\a_3$ as the set of all the points $\x$ which take the form
$$
\x = \lambda_0 \a_0 + \lambda_1 \a_1 + \lambda_2 \a_2 + \lambda_3 \a_3, \; \sum_{i=0}^3 \lambda_i = 1,
$$
with $\lambda_0 \leq 0$, $\lambda_i \geq 0$ for $1 \leq i \leq 3$. It is usual to call $\a_0$ the fictitious vertex of $T$, while 
$\a_1$, $\a_2$ and $\a_3$ are called the real vertices.  It is worth noting that $T$ is closed and convex. The tetrahedron $S_T$,  associated to $T$, is the convex hull of the points $\a_0$, $\a_1$, $\a_2$ and $\a_3$.  The {\it altitude vector} of $T$ is $\h_{T} = \pi_T \a_0 - \a_0$,  where $\pi_T \a_0$ is the  orthogonal projection of $\a_0$ on the affine plane containing $\a_1$, $\a_2$ and $\a_3$
 and  separating $T$ and $S_T$.  Notice that
 \begin{equation}\label{prop_alt}
 \forall \x \in T \cap S_T,  \; \h_T.(\x-\a_0) =  |\h_T|^2.
 \end{equation}
   \begin{figure}[ht] 
\centering
  \includegraphics[width=5.5cm]{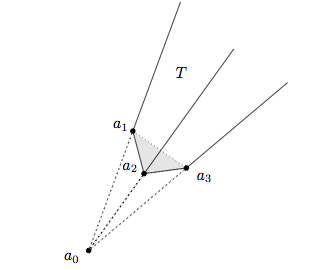}
\caption{An example of an infinite tetrahedron. }
\label{fig4}
\end{figure}
 Let us go back now to the decomposition \eqref{decompo_inft}. We should also note
 that $M$, the number of infinite tetrahedra, is not intended to be large. This is just a domain decomposition
 of a $\ominft$ in which the subdomains are  infinite tetrahedra and are fixed once for all. In practice, $M$ is often small ($M=3, 4,...$).  An example is illustrated in Figure \ref{big_tetrahedron_fig} where $\omz$ is a big  tetrahedron centered at  the origin and 
 $\ominft$ is the union of $4$ infinite tetrahedra (see also section \ref{numres_sect} hereafter). Another possibility consists to choose $\omz$ as the octahedron $\{ \x = (x_1, x_2, x_3)\in \R^3 \bve |x_1|+|x_2|+|x_3| < R \}$ and $\ominft = \R^3 \backslash \overline{\omz}$ as the union of $8$ infinite tetrahedra (see Figure \ref{octahedra_fig}). 
  \begin{figure}[ht] 
\centering
  \includegraphics[width=7cm]{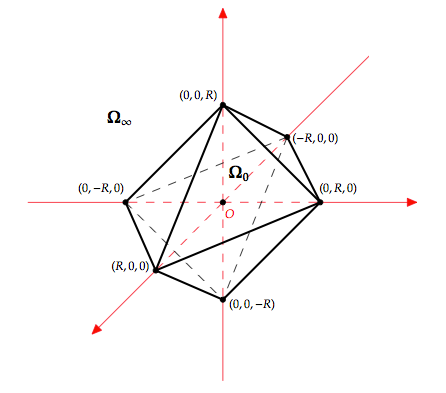}
\caption{A decomposition of $\R^3$ in which $\omz$ is the octahedron $\{ \x = (x_1, x_2, x_3)\in \R^3 \bve |x_1|+|x_2|+|x_3| < R \}$. Here $\ominft$ is the union of 8 infinite tetrahedra. }
\label{octahedra_fig}
\end{figure}
Subsequently, we denote by $S_i$, $1 \leq i \leq M$, the tetrahedron associated to $T_i$ and by $\h_i$ the altitude vector corresponding of $T_i$.  We have
$$
\overline{\omz} \cap \overline{\ominft}  = \cup_{i=1}^M (T_i \cap S_i).
$$
Set 
\begin{equation}
r_i(\x) = \frac{\h_i . \x}{|\h_i|^2} \mbox{ for } \x \in S_i \cup T_i. 
\end{equation}
Since $\zero$ is the fictitious vertex of each $T_i$ and in view of \eqref{prop_alt} it can  easily be proved that
$$
r_i(\x) \geq 1 \mbox{ for }  \x \in T_i, \; r_i(\x) \leq  1 \mbox{ for } \in S_i, \mbox{ and }  r_i(\x) =   1 \mbox{ for } \x \in T_i \cap S_i. 
$$
In terms of  the local barycentric coordinates in $S_i$, $(\lambda^{(i)}_0, \lambda^{(i)}_1, \lambda^{(i)}_2, \lambda^{(i)}_3)$, we can write
 $$
 r_i(\x) = 1-\lambda^{(i)}_0(\x)  = \sum_{k=1}^3 \lambda^{(i)}_k(\x) \mbox{ for } \x \in T_i \cup S_i.
 $$ 
The following continuity property holds true: if $T_i$ and $T_j$ are neighbors, then
\begin{equation}\label{neighbor_cont}
r_i(\x) = r_j(\x) \mbox{ for all } \x \in T_i \cap T_j.
\end{equation}
The local   polygonal inversion associated to
$T$ is defined as 
\begin{equation}\label{inversion_mapping}
\begin{array}{rrll}
\phi_i \; \; :& \; (S_i \cup T_i)\backslash \{\zero\}& \longrightarrow & (S_i \cup T_i)\backslash  \{\zero\}, \\ 
&\x &\mapsto&  \dps{ \frac{\x}{r_i(\x)^2}. }
\end{array}
\end{equation}
 Obviously $\phi_i$ is a bijection between $S_i$ and $T_i$. It is also an involution which preserves 
 $T_i \cap S_i$, that is $\phi_i(\x) = \x$ for $\x \in T_i \cap S_i$.  \\
Define now the {\it global polygonal inversion} $\phi$  from $\R^3 \setminus \{0\}$ into itself
and the  {\it  global polygonal radius} $r(.)$ as follows
$$
r(\x) = r_i(\x) \mbox{ and } \phi(\x) =\phi_i (\x)   \mbox{ for   all } \x \in T_i \cup S_i,  1\leq i \leq M.
$$
In virtue of property \eqref{neighbor_cont},  $r$ and $\phi$ are well defined and {\it continuous} on $\R^3$. Moreover,
$\phi$ maps $\ominft$ into $\omz$ and conversely. We have 
\begin{equation}
\phi(\x) = \x \mbox{ for all } \x \in \overline{\omz} \cap \overline{\ominft}. 
\end{equation}
Moreover, the exists two constants $c_1 > 0$ and $c_2 >0$ such that 
 $$
c|\x| \leq r(\x) \leq  c_2 |\x| \mbox{ for  all } \x \in \Rt, 
$$
In other words, $r(\x) \sim |\x|$ and  $|\phi(\x)| \sim {|\x|^{-1}} \mbox{ for } \x \in \Rt$. In the sequel, 
$\oms$ stands  for the image of $\ominft$ by the inversion $\phi$. From a strictly mathematical 
point of view $\oms = \omz \backslash \{\zero\}$. However, since $\omz$ and $\oms$ will be meshed differently. We will therefore deal with them separately. \\
 We now consider a family of pairs of 
triangulations $(\Tauh, \Taush)_h$  where 
\begin{itemize}
\item $(\Tauh)_h$ are  regular  triangulations of $\omm_0$ satisfying the usual conformity assumptions (see \cite{boulmezaoudm2an} or \cite{ciarlet}). In particular, elements of $\Tauh$ are supposed shape regular in the usual sense:  there exists a constant $c_0 > 0$ not depending on $h$ such that 
\begin{equation}\label{shapereg_tauh0}
\max_{K \in \Tauh \cup \Taush } \frac{h_K}{\rho_K} \leq c_0.
\end{equation}
Here $h_K$ and $\rho_K$ are respectively the diameter of $K$ and the diameter of the sphere inscribed inside of the tetrahedron $K$.  
\item $(\Taush)_h$ are  regular triangulations of $\oms$ which satisfies, besides \eqref{shapereg_tauh0},   the  following additional requirement: 
\begin{itemize}
\item for each $\Ks \in \Taush$, there exists $i \leq M$ such that $\Ks \subset S_i$ (in other words, $\Taush$ is a conforming union of triangulations of the subdomains $S_1$,..,$S_M$). 
\item the triangulations $(\Taush)_h $ are $\mu$-graded, where $\mu \in (0, 1]$ is a fixed parameter. That means that there exists three constants $\cst_1 > 0$, $\cst_2 > 0$ and $\cst_3 > 0$, not depending on $h$, such that
\begin{equation}\label{shapereg_tauh1}
\max_{K \in \Taussh} \frac{h_K}{d^{1-\mu}_K} \leq \cst_1 h, \;
\end{equation}
\begin{equation}\label{shapereg_tauh2}
\max_{K \in \Taush \backslash \Taussh} h_K \leq \cst_2 h^{1/\mu}, \;
\end{equation}
\begin{equation}\label{shapereg_tauh3}
\min_{\Ks \in \Taussh} d_{\Ks} \geq \cst_3  h^{1/\mu}, 
\end{equation}
where $\Taussh = \{ K \in \Taush \bve \zero \not \in K\}$ (elements not touching the origin),   $d_K = \inf_{\x \in K} \nrm{\x}$  for all $K \in \Taush$, and
$$
 h = h(\Tauh, \Taush)  = \max_{K \in \Tauh \cup  \Taush }  h_K,
$$
\end{itemize}
Conditions \eqref{shapereg_tauh1}, \eqref{shapereg_tauh2} and \eqref{shapereg_tauh3} mean that tetrahedra of 
 $\Taush$ which are adjacent to $\pt \oms\cap \pt \ominft$ have a size of order $h$, while those touching the fictitious vertex $\zero$ have a size  of order $h^{1/\mu}$.  Construction of graded meshes is detailed  in \cite{boulmezaoudm2an}. 
\item $\Tauh$ and $\Taush$ have the same  vertices, edges and faces on the common boundary $\omz \cap \ominft = \pt\omz=\pt \oms\backslash \{0\}$. 
 \end{itemize}
In the sequel,  given a function  $\v$ defined over   $\ominft$,  $\hv$ stands for
the function 
defined on $\oms$ as follows 
\begin{equation}\label{formule_hv}
 \hv(\xs) = \frac{1}{r(\xs)^{\gamma } }v(\phi(\xs)), \; \mbox{ for } \xs \in \oms. 
\end{equation}
with $\gamma > 0$ a parameter. Conversely, we have
\begin{equation}
 v(\x) = \frac{1}{r(\x)^{\gamma}\ }\hv(\phi(\x)), \; \mbox{ for } \x \in \ominft. 
\end{equation}
 Now, let $k \geq 1$ be a fixed integer and consider the finite dimensional space
  $$ 
 \Wh = \lbrace v \in \CC^{0}(\Rt) \bve  \forall K \in \mathcal{T}_{h}, v _{\vert K}\in (\P_{k})^3; \;\; \forall K^{*} \in \mathcal{T}_{h}^{*}, \hv_{\vert K^{*}}\in (\P_{k})^3, \, \hv(\zero)  = 0 \rbrace.
 $$
 We may observe that functions of $\Wh$ are piecewise polynomial in the FEM region $\omm_0$, but not in the IFEM
 region $\ominft$. The last observation is due to distorsion resulting from the composition with the inversion 
 and the multiplicative factor  involved in formula \eqref{formule_hv}. \\
Another observation concerns the behavior at large distances of functions belonging to $\Wh$. Let ${v}\in  \Wh$.  Then, $v \in H^1_{loc}(\R^3)^3$. Moreover, since $\hv (\zero) = 0$, we have 
 $$
|\hv(\xs)| \leq C_0 |\xs|, \mbox{ for all } \xs \in \oms,
$$
for some constant $C_0$, not depending on $\xs$.  It follows that  for all  $\x \in  \ominft$, we have 
$$
|v(\x)| = |r(\x)|^{-\gamma} |\hv(\phi(\x))| \leq C|\x|^{-\gamma } |\phi(\x)| \leq C |\x|^{-\gamma -1}.
$$
In similar way, we prove that
$$
|\grad v (\x)| \leq C |\x|^{-\gamma-2}.
$$
We deduce this
\begin{equation}
\gamma > -\frac{1}{2} \Longrightarrow \Wh\inclus W^1_0(\Rt)^3.
\end{equation}
This justifies the following assumption on $\gamma$: 
\begin{equation}
\gamma > -\frac{1}{2}. 
\end{equation}
The discrete problem writes: find $u_h \in  \Wh$ such that
\begin{equation}\label{discrete_pb}
\forall w_h \in \Wh, \; \int_{\R^3} \grad u_h. \grad w_h dx  = \int_{\omm} \Mg.\grad w_h dx. 
\end{equation}
The corresponding stray-field energy is given by
\begin{equation}\label{discrete_sf_energy}
\EESTH(u_h) =  \int_{\R^3} |\grad u_h|^2 dx =  \int_{\omm} \Mg.\grad u_h dx. 
\end{equation}
We have
 \begin{proposition}\label{error_estimate}
The discrete problem \eqref{discrete_pb} has one and only one solution $u_h \in \Wh$ and $\EESTH(u_h) \leq \EEST(u)$. If in addition, $\gamma   >  0$ and  $u \in W^{k+1}_{k+\gamma}(\Rt)$, then   
\begin{eqnarray}
\|u -u_{h}\|_{W^{1}_{0}(\Rt)^{3}}&\leq& C_1 h^{\tau k}\|u\|_{W^{k+1}_{k + \gamma}(\Rt)^{3}}, \label{er_estimate} \\
0 \leq \EEST(u) - \EESTH(u_h) &\leq& C_2  h^{\tau k}   \|\Mg\|_{L^2(\omm)}   \|u\|_{W^{k+1}_{k+\gamma}(\Rt)^{3}}, \label{er_energy}
\end{eqnarray}
where $C_1$ and $C_2$ are two constants not depending on $h$, $\Mg$ and $u$,  and 
\begin{equation}\label{formule_tau}
\tau = \min (\frac{\gamma}{\mu k}, 1). 
\end{equation}
 \end{proposition}
Proposition \ref{error_estimate} states in particular that if $u \in W^{k+1}_{k+\gamma}(\Rt)$ and if the mesh of $\oms$ is graded enough  ($\mu \leq \frac{\eta}{k}$), then the error is similar to that held in the finite element method in bounded domains, that is 
 \begin{equation}\label{estima_optimal1}
\|u -u_{h}\|_{W^{1}_{0}(\Rt)^{3}}\leq C h^{k}\|u\|_{W^{k+1}_{k+\gamma}(\Rt)^{3}}. 
\end{equation}
We also have
\begin{equation}\label{estima_optimal2}
 0 \leq \EEST(u) - \EESTH(u_h) \leq C_2  h^{k}  \|\Mg\|_{L^2(\omm)} \|u\|_{W^{k+1}_{k+\gamma}(\Rt)^{3}}.
\end{equation}
When smoothness of $u$ is only local, we have this 
\begin{proposition}\label{error_estimate_bis}
Assume  $u_{|\Rt \backslash \bomm}  \in  W^{k+1}_{k+\gamma}(\Rt \backslash \bomm)$ and $u_{| \omm} \in  H^{k+1}(\omm)$. Assume also that $\omm \subset \omz$ and  that  $\Tauh_{|\omm}$ is a triangulation of $\omm$.   Then, 
\begin{equation}\label{er_estimate_bis}
\|u -u_{h}\|_{W^{1}_{0}(\Rt)^{3}}\leq C_1 (h^{\tau k}\|u\|_{W^{k+1}_{k + \gamma}(\ominft)} + h^{k}\|u\|_{H^{k+1}(\omm)}+h^{k}\|u\|_{H^{k+1}(\omz \backslash \bomm)}),
\end{equation}
\begin{equation}\label{er_energy_bis} 
0 \leq  \EEST(u) - \EESTH(u_h) \leq C_2    \|\Mg\|_{L^2(\omm)}   (h^{\tau k} \|u\|_{W^{k+1}_{k + \gamma}(\ominft)} + h^{k}\|u\|_{H^{k+1}(\omm)}+h^{k}\|u\|_{H^{k+1}(\omz \backslash \bomm)}).
\end{equation}
with $\tau$ given by \eqref{formule_tau}. 
\end{proposition}
From Proposition \ref{regul_propo}, we know that if $\div \Mg \in L^2(\omm)$ and $ \Mg. \n \in H^{1/2}(\pt \omm)$, then $u_{|\Rt \backslash \bomm}  \in  W^{2}_{2}(\Rt \backslash \bomm)$ and $u_{| \omm} \in  H^{2}(\omm)$. With $k=1$ (P1 like elements) and $\gamma = 1$, we get the error estimate
\begin{equation} \label{er_on_u_ter}
\|u -u_{h}\|_{W^{1}_{0}(\Rt)^{3}}\leq C h (\|u\|_{W^{2}_{2}(\Rt\backslash \bomm)} + \|u\|_{H^{2}(\omm)}),
\end{equation}
\begin{equation} \label{er_energy_ter}
 0 \leq \EEST(u) - \EESTH(u_h) \leq C_2    \|\Mg\|_{L^2(\omm)}   h ( \|u\|_{W^{2}_{2}(\Rt\backslash \bomm)} + \|u\|_{H^{2}(\omm)}),   \end{equation}
for any $\mu \in (0, 1]$. This estimate  is similar to the usual finite element error for elliptic problems in 
bounded domain.  \\
\begin{ourproof}{of Propositions \ref{error_estimate} and \ref{error_estimate_bis}}
 Observe first that $u$ is also solution of the minimization problem
$$
\min_{v \in W^{1}_{0}(\Rt)^{3}} F(v), \; \mbox{ with } F(v)=  \frac{1}{2} \int_{\Rt}|\grad v|^2 - \int_{\Rt} \Mg.\grad v dx, 
$$
and, by virtue of \eqref{forma_weak}, we have 
$$
F(u) = - \frac{1}{2} \int_{\Rt}|\grad u|^2 dx  = -\EEST(\Mg), 
$$
Similarly, the approximate solution $u_h$ is solution of 
$$
\min_{\v_h \in \Wh} F(v_h), 
$$
and, in view of \eqref{discrete_pb}, we have 
$$
F(u_h) =  - \frac{1}{2} \int_{\Rt}|\grad u_h|^2 dx   = -   \EESTH(u_h), 
$$
Since $\Wh \subset W^{1}_{0}(\Rt)^{3}$, we deduce that $F(u_h)  \geq F(u)$. Thus, $\EESTH(u_h) \leq \EEST(u)$. \\
Now, C\'ea's lemma gives
$$
\|u -u_{h}\|_{W^{1}_{0}(\Rt)^{3}} \leq C_1 \inf_{w_h \in \Wh} \|u -w_{h}\|_{W^{1}_{0}(\Rt)^{3}}.
$$
for some constant $C_1$ not depending on $u$ nor on $h$.  In \cite{boulmezaoudm2an}, the following estimate is proven:
$$
\inf_{w_h \in \Wh} \|u-w_{h}\|^2_{W^{1}_{0}(\Rt)^{3}} \leq C (h^{2\tau k}  \|u\|^2_{W^{k+1}_{k + \gamma}(\ominft)}   +  h^{2k}  \sum_{K \in \Tauh} \|u\|^2_{H^{k+1}(K)}).  
$$
We easily get estimate \eqref{er_estimate} and \eqref{er_estimate_bis}.  In addition, we have
$$
 0 \leq \EEST(u) - \EESTH(u_h) = \frac{1}{2} \int_{\Rt}|\grad u|^2  dx - \frac{1}{2} \int_{\Rt}|\grad u_h|^2 dx = \frac{1}{2} \int_{\Rt}\Mg. (\grad u - \grad u_h) dx. 
$$
Thus, 
$$
  \EEST(u) - \EESTH(u_h)\leq  \frac{1}{2}  \|\Mg\|_{L^2(\omm)}. \|\grad u - \grad u_h\|_{L^2(\omm)}. 
$$
Combining with \eqref{er_estimate} gives \eqref{er_energy}.  Estimates \eqref{er_estimate_bis} and \eqref{er_energy_bis} are obtained by the same argument. 
\end{ourproof}

\section{Numerical results}\label{numres_sect}
The task of this section is to show some numerical results obtained with a 3D code writting
for solving \eqref{main_equa} with the following parameters: $k=1$ (P1 like elements). We use the following domain decomposition of $\R^3$:
\begin{itemize}
\item $\omz$ is the (big) tetrahedra whose vertices are
\begin{equation}\label{big_tetrahedron_vrtc}
\begin{array}{rclrcl} 
\a_1 &=& \dps{ \RZ (\frac{\sqrt{8}}{3}, 0, -\frac{1}{3})}, &  \a_2 &=& \dps{  \RZ (-\frac{\sqrt{2}}{3}, \sqrt{\frac{2}{3}},  -\frac{1}{3}), } \\
  \a_3 &=& \dps{  \RZ (-\frac{\sqrt{2}}{3},-  \sqrt{\frac{2}{3}},  -\frac{1}{3})},  &  \a_4 &=& \dps{ \RZ (0, 0,   1),  } 
\end{array}
\end{equation}
where $\RZ > 0$ is a size parameter (see Figure \ref{big_tetrahedron_fig}).
\item $\ominft = \R^3 \backslash \bomz$ is decomposed as the union of four infinite simplices $T_i$, $1 \leq i \leq 4$,  with the origin as a common fictitious vertex. The three real vertices $T_i$, $1 \leq i \leq 4$, are $(\a_j)_{1 \leq j \ne i \leq 4}$ (the bounded faces of $T_i$, $1 \leq i \leq 4$, are the faces of $\omz$). 
\end{itemize}
The code we write does not depend on the considered configuration. It only requires that
$\omm \subset \omz$. 
\begin{figure}[ht] 
\centering
   \includegraphics[width=6cm]{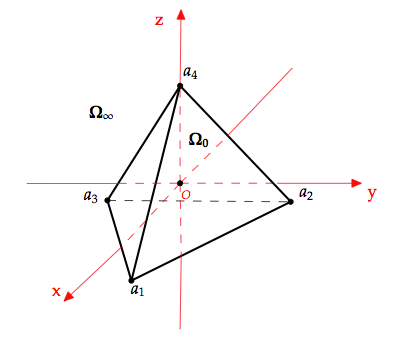}
\caption{The decomposition $\R^3 = \bominft \cup \bomz$ is used in implementation. Here $\bomz$  is a big tetrahedron whose vertices are given by formula \eqref{big_tetrahedron_vrtc}. } \label{big_tetrahedron_fig}
\end{figure}
In all the tests, we choose
$$
\gamma = 1.
$$
\subsection*{Numerical example 1 (homogeneously magnetized sphere)}
We consider the case of a ball 
$\omm =\{\x \in \R^3\bve  |\x| < \RAY\}$.  If $\Mg$ is constant, that is $\Mg = \mz$ for some unit vector field $\mz \in \R^3$, then 
the solution of \eqref{main_equa} is given by
\begin{equation}
u(\x) = \left\{
\begin{array}{ll}
\dps{ \frac{1}{3}{ \mz.\x } }& \mbox{ if } |\x| <  \RAY, \;  \\
\dps{\frac{ \RAY^3}{3}  \frac{\mz.\x}{|\x|^3} } & \mbox{ if } |\x| \geq   \RAY.
\end{array}
\right.
\end{equation}
In computational tests, we choose $\mz = (0, 0, 1)$, $r_0 = 0.5$ and $\RZ = 4$. In order to check the convergence 
of the method, we measure the following errors 
$$e_{0}(u) = \frac{\|u_{h} - u\|_{W^{0}_{-1}(\Rt)}}{\|u\|_{W^{0}_{-1}(\Rt)}},  \; e(\EEST) = \frac{ |\EEST(u) -\EESTH(u_h)|}{|\EEST(u)|}. $$
%e (\nabla u) = \frac{\|\nabla u_{h} - \nabla u\|_{L^{2}(\Rt)^3}}{\|\nabla u\|_{L^{2}(\Rt)^3}}. $$
In the context of this example, the exact energy is given by 
\begin{equation}
\EEST(u) =  \frac{1}{2} \int_{\R^3} |\grad u|^2 dx =  - \frac{1}{2} \int_{\omm} \Mz.\vH dx  =  \frac{ 2 \pi |\mz|^2}{9}  \RAY^3 = 0.08723. 
\end{equation}
Obviously, $u_{|\omm} \in H^2(\omm)$. One can also check that $u_{|\Rt \backslash \bomm} \in W_2^2(\Rt \backslash \bomm)$. However,
 $ u \not \in W^2_2(\Rt)$ since $[\frac{\pt u}{\pt n}] = -\mz.\n \ne 0$ on $\pt \omm$. According to estimates  \eqref{er_energy_ter} and \eqref{er_energy_ter}
 errors $e_{0}(u)$ and $ e(\EEST)$ decrease as $h$ (for any gradation parameter $\mu$).  Table \ref{tab_exConstSphere} and Figure \ref{enrg_exConstSphere} 
 display these errors versus  $h$ for several values of  $\mu$. We may observe that $e_{0}(u)$ decreases as
 $h^{0.96}$ while $ e(\EEST)$ decreases as $h^{1.33}$.  The errors are
  essentially the same for $\mu=1$, $\mu = 0.7$ and $\mu =0.5$.  This is in accordance with estimates \eqref{er_on_u_ter}  and \eqref{er_energy_ter}.  
In Figure \ref{figm1} the approximate solution and the exact one  are displayed  versus $r$ when $x=y=0$.  It can be seen by a visual comparison 
 that these solutions are very close although the discontinuity of the normal component of $ \vH = -\grad u$ across
  the interface $\pt \omm$.
  \begin{table}[htp]
\begin{center}
\begin{tabular}{|c|c|c c c |ccc |}
  \hline
& & \multicolumn{3}{c|}{$\mu$} &  \multicolumn{3}{c|}{$\mu$}   \\ 
 \cline{3-8} 
   DoF & $h$&  1 & 0.7 & 0.5 & 1 & 0.75 & 0.5  \\  
 \cline{3-8}
    &  & \multicolumn{3}{c|}{$e_{0}(u)$ }&\multicolumn{3}{c|}{$e(\EEST)$  (energy error)}\\
  \hline
  875  & 1.131 & 0.292&0.284&0.280&   0.532 &0.523&0.517\\
   6750 & 0.565 &0.145&0.139&0.135&  0.322&0.315&0.312 \\
  22625  & 0.377 &0.101&0.098& 0.097&  0.157&0.152& 0.150\\
  53500 &0.282 &0.076&0.074&0.073&   0.127&0.124&0.123 \\
   104375  & 0.226 &0.064& 0.062&0.062&  0.089& 0.086&0.085 \\
   180250  &0.188 &  0.052&0.051&0.051&   0.073&0.072&0.071  \\
  427000&0.141 &0.040& 0.039& 0.039 &   0.043& 0.042& 0.041 \\
  833750 &0.113 &0.032&0.031&0.031&   0.025&0.024&0.024  \\
    \hline
    \multicolumn{2}{|c|}{The log. slope}&0.96 &0.96 &0.96&  1.33 &1.34 &1.33 \\ 
  \hline
\end{tabular} 
\end{center}
\caption{(Example 1) The relative errors $e_{0}(u)$  and $ e(\EEST)$ ($\RAY=0.5$, $\gamma = 1$, and  $\RZ = 4$).}
\label{tab_exConstSphere}
\end{table}
\begin{figure}[!h]
\centering
\includegraphics[width= 12cm]{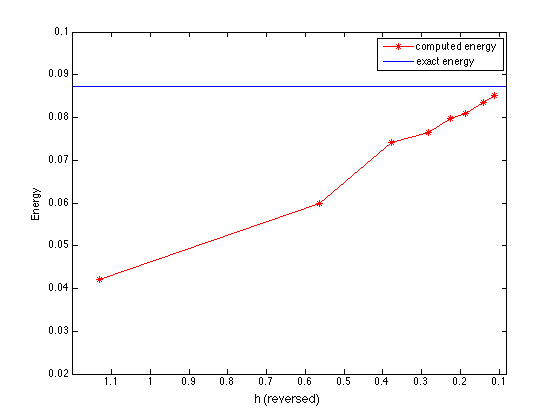}
\caption{(Example 1) the exact and the approximate energy versus $h$ with $\mu = 0.5$, $\gamma = 1$, $\RZ = 4$ and $\RAY=0.5$.  }
 \label{enrg_exConstSphere} 
\end{figure}
 \begin{figure}[!h]
\centering
\includegraphics[width= 12cm]{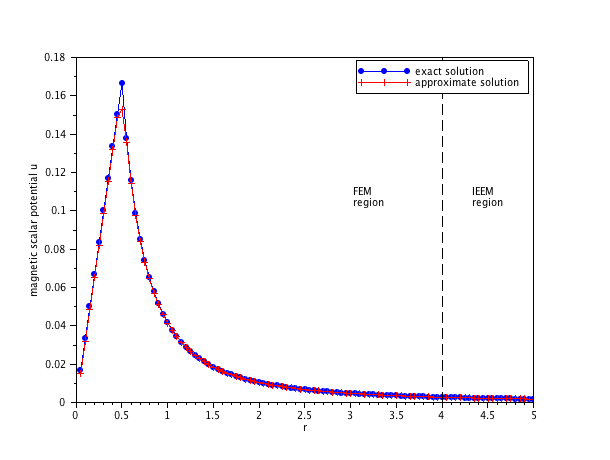}
\caption{(Example 1) the exact and the approximate  scalar potential versus $r=|\x|$ when $x=y=0$ and $z \geq 0$   ($\mu =0.5$, $\gamma = 1$, $\RZ = 4$, $r_0 = 0.5$ and $DoF =833750$). }
 \label{figm1}
\end{figure}

\subsection{Numerical example 2 (non homogeneously magnetized sphere)}
In this second example, we consider the case of a non homogeneous magnetization 
of a sphere   $\omm =\{\x \in \R^3\bve  |\x| < \RAY\}$. More precisely,   $\Mg$ is  the unit vector field 
\begin{equation}\label{unitMg_ball}
\Mg = (\cos \theta) \e_\varphi + (\sin \theta) \e_\theta \mbox{ in } \omm. 
\end{equation}
Here $(r, \varphi, \theta)$, $r \geq 0$, $0\leq \varphi \leq 2 \pi$, $0 \leq \theta \leq \pi$, denote the spherical coordinates and $(\e_r, \e_\varphi, \e_\theta)$ the corresponding unit vectors.   
In that case, the solution can be obtained explicitly (see the appendix):
\begin{equation}
u(\x) = \left\{
\begin{array}{ll}
\dps{ - \frac{2z}{9} +   \frac{2z}{3} \ln (\frac{|\x|}{ \RAY})}& \mbox{ if } |\x| \leq   \RAY, \;  \\
\dps{ -  \frac{2  \RAY^3 z}{9 |\x|^3} }  & \mbox{ if } |\x| \geq   \RAY.
\end{array}
\right.
\end{equation}
This solution belongs to $W^2_2(\R^3)$, as forecasted in Proposition \ref{regul_propo} (indeed,
 $\div \Mg \in L^2(\omm)$ and $\Mg.\n=0$ on $\pt \omm$).   In our numerical experiences,
  we fixe $r_0 = 0.5$ and $\RZ = 6$. The exact stray-field energy is given by
\begin{equation}
\En_{sf}(u) =  \frac{1}{2} \int_{\R^3} |\grad u|^2 dx   = \frac{16}{81} \pi \RAY^3= 0.0776. 
\end{equation}
According to Proposition  \ref{regul_propo}, no gradation is needed to get optimal convergence. 
More precisely, estimates \eqref{estima_optimal1} and
 \eqref{estima_optimal2} hold true for any gradation paramater $\mu \leq 1$ (here $k=1$ and $\gamma=1$), that is
 \begin{eqnarray}
\|u -u_{h}\|_{W^{1}_{0}(\Rt)^{3}}&\leq& C h \|u\|_{W^{2}_{2}(\Rt)^{3}},  \\ \label{estima_optimal1_ex2}
0 \leq  \EEST(u) - \EESTH(u_h)& \leq &C_2  h \|u\|_{W^{2}_{2}(\Rt)^{3}}. \label{estima_optimal2_ex2}
\end{eqnarray}
Table \ref{ex2_tab_errors} shows  these relative errors versus $h$.  The behavior of the discrete
stray-field energy $\EESTH(\Mg)$ is also displayed in Figure \ref{enrg_nonh_sphere}, while the approximate
and the exact solutions are displayed versus $ r = |\x|$ in Figure \ref{exa_appr_ex2} (when $x = y = z$).
We can observe that the energy converges as $h^{2.5}$. This superconvergence of  energy is not foreseen in estimate \eqref{er_energy}
 and has not been proved. We conjecture that this superconvergence of the energy holds when  $u \in W^2_2(\Rt)$ (or, equivalenty, when 
 $\Mg \in H(\div; \omm)$ and $\Mg.\n =0$ on $\pt \omm$).   
%\begin{center}
%\begin{tabular}{|c|c|c c c |ccc |}
%  \hline
%& & \multicolumn{3}{c|}{$\mu$} &  \multicolumn{3}{c|}{$\mu$}   \\ 
% \cline{3-8} 
%   DoF & $h$&  1 & 0.7 & 0.5 & 1 & 0.75 & 0.5  \\  
% \cline{3-8}
%    &  & \multicolumn{3}{c|}{$e_{0}(\u)$ }&\multicolumn{3}{c|}{$e (\nabla) $}\\
%  \hline
%  875  & 1.131 & 0.258&0.250&0.247&0.402&0.397&0.396\\
%   6750 & 0.565 &0.153&0.149&0.148&0.274&0.271&0.271   \\
%  22625  & 0.377 &0.101&0.099& 0.099&0.188& 0.186& 0.186 \\
%  53500 &0.282 &0.073&0.072&0.071&0.141&0.140&0.140  \\
%   104375  & 0.226 &0.061& 0.060&0.059&0.124&0.124&0.123  \\
%   180250  &0.188 &  0.049&0.048&0.048&0.106&0.105& 0.105  \\
%  427000&0.141 &0.037& 0.036& 0.036&0.088&0.088&0.087 \\
%  833750 &0.113 &0.029&0.029&0.029&0.076&0.076&0.076\\
%    \hline
%     \multicolumn{2}{|c|}{The log. slope}&0.948 &0.935 &0.929 &0.723 &0.717 &0.716 \\
%  \hline
%%    \multicolumn{2}{|c|}{The log. slope}&0.65 &0.74 &0.72 &0.68 &0.44 &0.69 &0.82 & 0.81\\
%%  \hline
%\end{tabular} 
%\end{center}
%\caption{The relative errors $e_{0}(\u)$  and $ e(\nabla)$ for $\gamma = 1, pform = 4$ (example 2).}
%\label{tab1}
%\end{table}
%
  \begin{table}[htp]
\begin{center}
\begin{tabular}{|c|c|c c c |ccc |}
  \hline
& & \multicolumn{3}{c|}{$\mu$} &  \multicolumn{3}{c|}{$\mu$}   \\ 
 \cline{3-8} 
   DoF & $h$&  1 & 0.7 & 0.5 & 1 & 0.75 & 0.5  \\  
 \cline{3-8}
    &  & \multicolumn{3}{c|}{$e_{0}(u)$ }&\multicolumn{3}{c|}{$e(\EEST)$}\\
  \hline
  875  & 1.131 & 0.258&0.250&0.247& 0.623 &0.619&0.615\\
   6750 & 0.565 &0.153&0.149&0.148&  0.325&0.322&0.320 \\
  22625  & 0.377 &0.101&0.099& 0.099& 0.151&0.149& 0.148\\
  53500 &0.282 &0.073&0.072&0.071& 0.084&0.083&0.082 \\
   104375  & 0.226 &0.061& 0.060&0.059&  0.054& 0.053&0.052 \\
   180250  &0.188 &  0.049&0.048&0.048&  0.033&0.032&0.032  \\
  427000&0.141 &0.037& 0.036& 0.036& 0.009& 0.008& 0.008 \\
  833750 &0.113 &0.029&0.029&0.029&   0.001&0.002&0.002\\
    \hline
     \multicolumn{2}{|c|}{The log. slope}&0.95 &0.94 &0.93 &2.79 &2.49 &2.49 \\
  \hline
%    \multicolumn{2}{|c|}{The log. slope}&0.65 &0.74 &0.72 &0.68 &0.44 &0.69 &0.82 & 0.81\\
%  \hline
\end{tabular} 
\end{center}
\caption{ (Example 2) The relative errors $e_{0}(u)$  and $e(\EEST)$  ($\gamma = 1$, $\RZ=6$ and $\RAY=0.5$).} \label{ex2_tab_errors}
\end{table}

\begin{figure}
\centering
\includegraphics[width= 12cm]{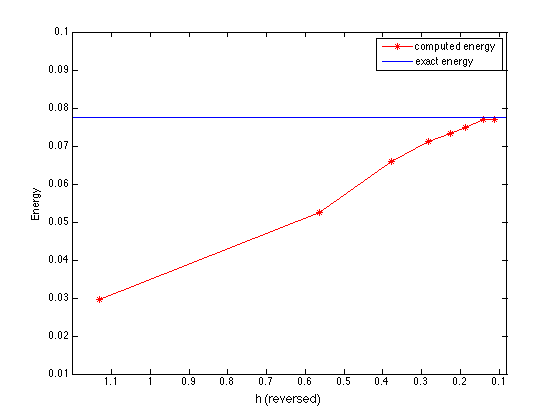}
\caption{(Example 2)  The exact and the approximate energy versus $h$ with $\mu = 0.5$, $\RAY = 0.5$, $\gamma = 1$ and $\RZ = 6$.  }
 \label{enrg_nonh_sphere}
\end{figure}
\begin{figure}
\centering
\includegraphics[width= 12cm]{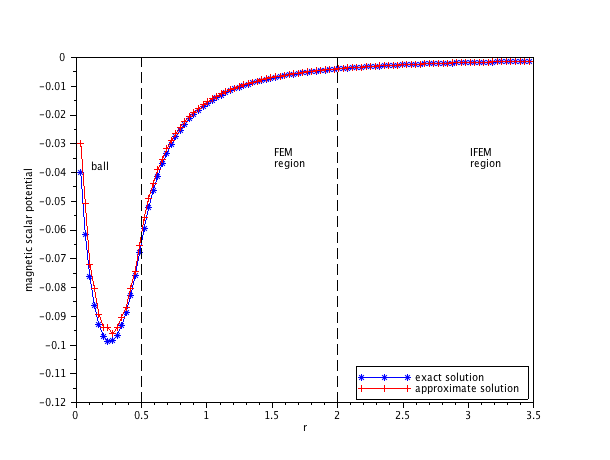}
\caption{(Example 2) The exact and approximate magnetic scalar potential versus $r = |\x|$ (with $x=y=z$). Here $\mu = 0.5, \gamma = 1,  \RZ = 6$ and $DoF =833750$. }
%revoir a nouveau l'intersection a r = 2 (ifem-non-ifem)
 \label{exa_appr_ex2}
\end{figure}

\subsection*{Numerical example 3}
As a last benchmark, we consider a homogeneously magnetized unit cube: $\Mg = (0, 0, 1)$
and $\omm = [-1/2, 1/2]^3$ (see, e. g., \cite{ClaasExl}). The stray-field energy in this case is given by
\begin{equation}
\EEST(u) = \frac{1}{6}. 
\end{equation}
 In Table \ref{tab_exemp3},
error on the energy versus $h$ is displayed for several values of
$\mu$. In figure \ref{ex3_energy}, we show  the evolution of the energy 
 versus $h$. 
Here also we may observe that the approximate energy converges quickly to
the exact one.  This  superconvergence  can clearly be seen  in 
Figure \ref{ex3_energy}. 
\begin{table}[htp]
\begin{center}
\begin{tabular}{|c|c|ccc|}
  \hline
& & \multicolumn{3}{c|}{The relative error of energy}   \\ 
 \cline{3-5} 
   DOF & $h$& \multicolumn{3}{c|}{$\mu$ }\\
 \cline{3-5}
    & &  1 & 0.7 & 0.5   \\ 
  \hline
  875  & 1.697 &0.468 &0.463&0.460\\
   6750 & 0.848 &0.364&0.360&0.358 \\
  22625  & 0.565 &0.298&0.295& 0.294\\
  53500 &0.424 &0.243&0.241&0.240 \\
   104375  & 0.339 &0.205& 0.203&0.203 \\
   180250  &0.282 & 0.134&0.133&0.133  \\
  427000&0.212 &0.013& 0.013& 0.012 \\
  \hline
   %\multicolumn{2}{|c|}{The  log. slope}&0.6980 &0.6950 &0.6914\\
%  \hline
\end{tabular} 
\end{center}
\caption{The relative error of energy  for $\gamma = 1$ (example 3)}
\label{tab_exemp3}
\end{table}

 \begin{figure}[!h]
\centering
\includegraphics[width= 10cm]{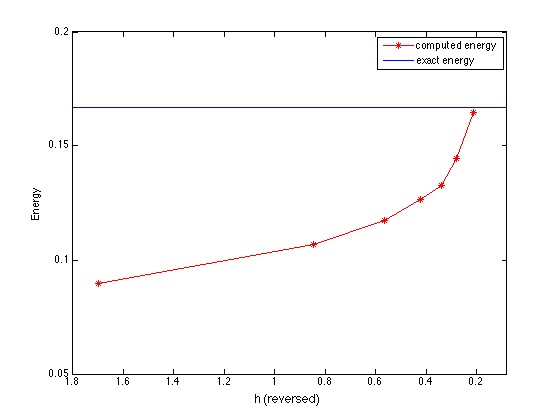}
\caption{ (Example 3) The computed energy of an homogeneously magnetized unit cube.}
 \label{ex3_energy}
\end{figure}
\appendix
\section{Solving the problem in the case of an inhomogeneously magnetized  ball (numerical example 2)}
The resolution of the system in the case of a ball $B_\RAY$ and $\Mg$ given by \eqref{unitMg_ball} can be done by means of a decomposition on  spherical harmonics  $(Y_\ell^m)_{\ell \geq 0, -\ell \leq m \leq \ell}$ (which is orthonormal with respect to the inner product in $L^2(\St)$).  We have outside the ball  $B_\RAY$:
$$
\Delta u = 0 \mbox{ in }\R^3 \backslash \overline{B}_\RAY. 
$$
Developping $u$ on the basis of spherical harmonics gives (see, e. g., \cite{booknedelec}):
$$
u(\x) = \sum_{\ell = 0}^{+\infty} \sum_{m=-\ell}^\ell A_\ell^m \left(\frac{r}{\RAY}\right)^{-\ell-1} Y_\ell^m( \varphi, \theta). 
$$
where  $(A_\ell^m)_{\ell \geq 0, -\ell \leq m \leq \ell}$ is a sequence of complex coefficients. On the other hand, we have in the interior of the ball
$$
\Delta u = \div \Mg = 2 \frac{\cos \theta}{r} = \frac{\alpha}{r}  Y_1^0(\varphi, \theta)  \mbox{ in } \overline{B}_\RAY, \; \; \mbox{ with } \alpha =  4 \sqrt{\frac{\pi}{3}}.  
$$
Writing  
$$
u (\x)  = \sum_{\ell = 0}^{+\infty} \sum_{m=-\ell}^\ell u_\ell^m(r)  Y_\ell^m ( \varphi, \theta), \mbox{ in } \overline{B}_\RAY, 
$$
gives: 
$$
\frac{1}{r^2} \frac{d}{dr } (r^2 \frac{d u_\ell^m }{dr }) (r) - \frac{\ell (\ell + 1) }{r^2} u_\ell^m(r)  =  \frac{\alpha }{r}  \delta_{\ell, 1}  \delta_{m, 0} \mbox{ for all } \ell \geq 0 \mbox{ and } -\ell \leq m \leq \ell, 
$$
where $\delta_{i, j}$, $i \in \N$, $j \in \N$,  denotes the usual  Kronecker delta. The solutions of this equations are of the form
$$
u_\ell^m (r)  = B_\ell^m  \left(\frac{r}{\RAY}\right)^{\ell} + C_\ell^m  \left(\frac{r}{\RAY}\right)^{-\ell-1} + \frac{\alpha}{3} r \ln( \frac{r}{\RAY}) \delta_{\ell, 1} \delta_{m, 0},
$$
where  $B_\ell^m$ and $C_\ell^m$ are constants.   Since
$u \in W^1_0(\R^3)$ we deduce that $u_{B_\RAY} \in H^1(B_\RAY)$. Necessarily 
$C_\ell^m = 0$ for all $\ell \geq 0$ and $|m| \leq \ell$. Since $[u]_{\pt \omm} = 0$, we deduce that
$ B_\ell^m = A_\ell^m$ for all $\ell \geq 0$ and $|m| \leq \ell$. In addition, 
$$
\left[ \frac{\pt u}{\pt r} \right]_{\pt \omm} =  -  \Mg . \n = -\Mg . \e_r = 0.
$$
Thus, for all $\ell \geq 0$ and $|m| \leq \ell $ we have
$$
\ell \frac{B^m_\ell}{\RAY} +  \frac{ \alpha}{3}  \delta_{\ell, 1} \delta_{m, 0}   = -\frac{ \ell + 1 }{\RAY} A^m_\ell. 
$$
Thus, $A_l^m = B_l^m=0$ for $(\ell, m) \ne (1, 0)$ and  
$$
A_1^0 = B_1^0 =  - \frac{\alpha \RAY}{9}. 
$$
Thus, if $|\x| \geq \RAY$ then
$$
u(\x)  = A_1^0    \left(\frac{\RAY}{r}\right)^{2} Y_1^0 =  - \frac{\RAY^3}{9 r^2} \alpha Y_1^0 = - \frac{2 \RAY^3 \cos \theta }{9 r^2} = - \frac{2 \RAY^3 z }{9 r^3}.  
$$
and 
$$
\|\grad u\|^2_{L^2(\R^3 \backslash \overline{B}_\RAY)} =    \frac{32 \pi}{243}\RAY^3. 
$$
If $|\x| \leq \RAY$ then
$$
u (\x)  = (B_1^0    \frac{r}{\RAY} +\frac{\alpha}{3} r \ln( \frac{r}{\RAY})) Y_1^0   = ( - \frac{r}{9} + \frac{r}{3} \ln( \frac{r}{\RAY})) \alpha Y_1^0=  \frac{2 r  \cos \theta }{9} ( - 1 + 3  \ln( \frac{r}{\RAY})). 
$$
Thus, 
$$
u  =  \frac{2 z }{9} ( - 1 + 3  \ln( \frac{r}{\RAY})), 
$$
and
$$
\|\grad u\|^2_{L^2( \overline{B}_\RAY)}  =\frac{64 \pi}{243}\RAY^3. 
$$
Thus, the energy of the corresponding stray-field is 
$$
\En_{sf} (u) =  \frac{1}{2} \int_{\RAY^3} |\grad u|^2 dx    = \frac{16}{81} \pi \RAY^3.
$$

$\;$\\
{\bf Aknowlegement}. \\
This work was partially supported by a public grant as part of the
Investissement d'avenir project, reference ANR-11-LABX-0056-LMH,
LabEx LMH.

$\;$\\
{\bf Declarations}. \\
{\it Conflict of interest}: The authors declare no competing interests.

\bibliographystyle{plain}
\bibliography{micromag.bib}

\begin{thebibliography}{10}

\bibitem{alliot_these}
F.~Alliot.
\newblock {\em Etude des \'equations stationnaires de Stokes et Navier-Stokes
  dans des domaines ext\'erieurs}.
\newblock PhD Thesis, ENPC, Paris, 1998.

\bibitem{amrouche1}
C.~Amrouche, V.~Girault, and J.~Giroire.
\newblock Weighted {S}obolev spaces for {L}aplace's equation in {$\bold R\sp
  n$}.
\newblock {\em J. Math. Pures Appl. (9)}, 73(6):579--606, 1994.

\bibitem{Berkov1}
D.~V. Berkov, K.~Ramst\"ock, and A.~Hubert.
\newblock Solving micromagnetic problems: toward and optimal numerical method.
\newblock {\em Phys. Stat. Sol (a)}, 137:207--225, 1993.

\bibitem{BlueShen}
J.~L. Blue and M.~R. Scheinfein.
\newblock Using multipoles decreases computation time for magnetic self-energy.
\newblock {\em IEEE Trans. Magn.}, 27:4778--4780, 1991.

\bibitem{boulmezaoudm2an}
T.~Z. Boulmezaoud.
\newblock Inverted finite elements: a new method for solving elliptic problems
  in unbounded domains.
\newblock {\em M2AN Math. Model. Numer. Anal.}, 39(1):109--145, 2005.

\bibitem{carstensen2001}
C.~Carstensen and A.~Prohl.
\newblock Numerical analysis of relaxed micromagnetics by penalised finite
  elements.
\newblock {\em Numer. Math.}, 90(1):65--99, 2001.

\bibitem{ciarlet}
Ph.-G. Ciarlet.
\newblock {\em The finite element method for elliptic problems}.
\newblock North-Holland Publishing Co., Amsterdam, 1978.

\bibitem{ClaasExl}
A.~Class, L.~Exl, G.~Selke, A.~Drews, and Th. Schrefl.
\newblock Fast stray field computation on tensor grids.
\newblock {\em Journal of magnetism and magnetic materials}, 176(326), 2013.

\bibitem{ExlAuz}
L.~Exl, W.~Auzinger, S.~Bance, M.~Gusenbauer, F.~Reichel, and T.~Schrefl.
\newblock Fast stray field computation on tensor grids.
\newblock {\em J. Comput. Phys.}, 231(7):2840--2850, 2012.

\bibitem{KoehlerFred1}
D.~R. Fredkin and T.~R. Koehler.
\newblock Hybrid method for computing demagnetizing fields.
\newblock {\em IEEE Trans. Magn.}, 26:415--417, 1990.

\bibitem{giroire}
J.~Giroire.
\newblock {\em Etude de quelques probl{\`e}mes aux limites ext\'{e}rieurs et
  r\'{e}solution par \'{e}quations int\'{e}grales}.
\newblock Th\`{e}se de Doctorat d'Etat. Universit\'{e} Pierre et Marie Curie,
  Paris, 1987.

\bibitem{hanbao_poisson1999}
H.~Han and W.~Bao.
\newblock The discrete artificial boundary condition on a polygonal artificial
  boundary for the exterior problem of {P}oisson equation by using the direct
  method of lines.
\newblock {\em Comput. Methods Appl. Mech. Engrg.}, 179(3-4):345--360, 1999.

\bibitem{hanbao_siam2000}
H.~Han and W.~Bao.
\newblock Error estimates for the finite element approximation of problems in
  unbounded domains.
\newblock {\em SIAM J. Numer. Anal.}, 37(4):1101--1119, 2000.

\bibitem{hanouzet}
B.~Hanouzet.
\newblock Espaces de {S}obolev avec poids application au probl{\`e}me de
  {D}irichlet dans un demi espace.
\newblock {\em Rend. Sem. Mat. Univ. Padova}, 46:227--272, 1971.

\bibitem{KoehlerFred2}
T.~R. Koehler and D.~R. Fredkin.
\newblock Finite element methods for micromagnetism.
\newblock {\em IEEE Trans. Magn.}, 28:1239--1244, 1992.

\bibitem{labbe}
S.~Labb{\'e}.
\newblock Fast computation for large magnetostatic systems adapted for
  micromagnetism.
\newblock {\em SIAM J. Sci. Comput.}, 26(6):2160--2175, 2005.

\bibitem{longOng}
H.~Long, E.~Ong, Z.~Liu, and E.~Li.
\newblock Fast fourier transform on multipoles for rapid calculation of
  magnetostatic fields.
\newblock {\em IEEE Trans. Magn.}, 42:295--300, 2006.

\bibitem{LuskinMA}
M.~Luskin and L.~Ma.
\newblock Analysis of the finite element approximation of microstructure in
  micromagnetics.
\newblock {\em SIAM J. Numer. Anal.}, 29(2):320--331, 1992.

\bibitem{booknedelec}
J.-C. N{\'e}d{\'e}lec.
\newblock {\em Acoustic and electromagnetic equations}, volume 144 of {\em
  Applied Mathematical Sciences}.
\newblock Springer-Verlag, New York, 2001.
\newblock Integral representations for harmonic problems.

\bibitem{Popovi}
N.~Popovi{\'c} and D.~Praetorius.
\newblock Applications of {$H$}-matrix techniques in micromagnetics.
\newblock {\em Computing}, 74(3):177--204, 2005.

\bibitem{Prohlbook}
A.~Prohl.
\newblock {\em Computational micromagnetism}.
\newblock Advances in Numerical Mathematics. B. G. Teubner, Stuttgart, 2001.

\end{thebibliography}

\end{document}